\documentclass[a4paper,12pt,final]{amsart}

\usepackage{amssymb}
\usepackage{showkeys}
\usepackage[pdfauthor={Eugen Stumpf},%
pdftitle={Local stability analysis of differential equations with state-dependent delay},%
colorlinks=true]{hyperref}

\newtheorem{thm}{Theorem}[section]
\newtheorem{prop}[thm]{Proposition}

\newtheorem{cor}[thm]{Corollary}
\theoremstyle{definition}

\theoremstyle{remark}

\newtheorem{remark}[thm]{Remark}

\numberwithin{equation}{section}
\newcommand{\N}{\mathbb{N}}  
\newcommand{\R}{\mathbb{R}}  
\newcommand{\C}{\mathbb{C}}  
\DeclareMathOperator{\id}{id} 

\title[Stability analysis of state-dependent delay equations ]{Local stability analysis of differential equations with state-dependent delay}
\author{Eugen Stumpf}
\address{Department of Mathematics, University of Hamburg, Bundesstrasse 55, 20146 Hamburg, Germany}
\email{eugen.stumpf@math.uni-hamburg.de}
\urladdr{www.math.uni-hamburg.de/home/stumpf/index\_en.html}

{}
\copyrightinfo{2015}{Eugen Stumpf}

\PII{}

\begin{document}

\begin{abstract}
  In the present article, we discuss some aspects of the local stability analysis for a class of abstract functional differential equations. This is done under smoothness assumptions which are often satisfied in the presence of a state-dependent delay. Apart from recapitulating the two classical principles of linearized stability and instability, we deduce the analogon of the \textit{Pliss reduction principle} for the class of differential equations under consideration. This reduction principle enables to determine the local stability properties of a solution in the situation where the linearization does not have any eigenvalues with positive real part but at least one eigenvalue on the imaginary axis.
\end{abstract}
\maketitle
\section{Introduction}
Let $h>0$ and $n\in\N$ be fixed. Further, after choosing some norm $\|\cdot\|_{\R^{n}}$ in $\R^{n}$, let us denote by $C$ the Banach space of all continuous functions $\phi:[-h,0]\to\R^{n}$ provided with the norm $\|\phi\|_{C}:=\sup_{s\in[-h,0]}\|\phi(s)\|_{\R^{n}}$. Similarly, $C^{1}$ denotes the Banach space of all continuously differentiable $\phi:[-h,0]\to\R^{n}$ with the norm given by $\|\phi\|_{C^{1}}:=\|\phi\|_{C}+\|\phi^{\prime}\|_{C}$. For any continuous function $x:I\to\R^{n}$ defined on some interval $I\subset\R$ and any real $t\in\R$ with $[t-h,t]\subset I$, let $x_{t}\in C$ denote the so-called \textit{segment} of $x$ at $t$, that is, the function $x_{t}:[-h,0]\to\R^{n}$ defined by $x_{t}(\theta):=x(t+\theta)$ for all $-h\leq \theta\leq 0$.

In what follows, we consider the functional differential equation
\begin{equation}\label{eq: FDE}
  x^{\prime}(t)=f(x_{t})
\end{equation}
defined by some map $f:U\to\R^{n}$ from an open neighborhood $U\subset C^{1}$ of the origin in $C^{1}$ into $\R^{n}$ with $f(0)=0$. In doing so, we have in mind that Eq. \eqref{eq: FDE} represents a differential equation with a state-dependent delay in a more abstract form. In order to clarify this point, consider for simplicity the differential equation
\begin{equation}\label{eq: DDE}
  x^{\prime}(t)=\tilde{g}(x(t-r(x(t))))
\end{equation}
with a map $\tilde{g}:\R^{n}\to\R^{n}$ satisfying $\tilde{g}(0)=0$ and with a discrete state-dependent delay given by some function $r:\R^{n}\to[0,h]$. Defining the map $\tilde{f}:C^{1}\to\R^{n}$ by
\begin{equation*}
  \tilde{f}(\phi):=\tilde{g}(\phi(-r(\phi(0)))),
\end{equation*}
we see that Eq. \eqref{eq: DDE} can be written in the more abstract form
\begin{equation*}
  x^{\prime}(t)=\tilde{g}(x(t-r(x(t))))=\tilde{g}(x_{t}(-r(x_{t}(0))))=\tilde{f}(x_{t})
\end{equation*}
of Eq. \eqref{eq: FDE}. Hence, instead of studying the original differential equation \eqref{eq: DDE}, we may as well study Eq. \eqref{eq: FDE}.

The proposed transformation works also for many other differential equations with state-dependent delay. In addition, observe that in the discussed example the map $\tilde{f}$ could have been defined on the greater Banach space $C$ and not on $C^{1}$. Then Eq. \eqref{eq: FDE} would form a so-called \textit{retarded functional differential equation} as considered for example in Diekmann et al. \cite{Diekmann1995}. But in contrast to the constant delay case, the theory of retarded functional differential equations is in general not applicable in the presence of a state-dependent delay (see for instance Walther \cite{Walther2003}).

A \textit{solution} of Eq. \eqref{eq: FDE} is either a $C^{1}$-smooth function $x:[t_{0}-h,t_{e})\to\R^{n}$, $t_{0}<t_{e}\leq\infty$, such that $x_{t}\in U$ for all $t_{0}\leq t<t_{e}$ and $x$ satisfies \eqref{eq: FDE} as $t_{0}<t<t_{e}$, or a $C^{1}$-smooth function $x:\R\to\R^{n}$ such that $x_{t}\in U$ for each $t\in\R$ and Eq. \eqref{eq: FDE} is satisfied everywhere in $\R$. For instance, $x:\R\ni t\mapsto 0\in\R^{n}$ is a solution of Eq. \eqref{eq: FDE} in view of the assumption $f(0)=0$.

In order to get further solutions of Eq. \eqref{eq: FDE}, we shall make two standing smoothness assumptions on the map $f$ under consideration:
\begin{itemize}
  \item[(S1)] $f$ is continuously differentiable, and
  \item[(S2)] for each $\phi\in U$ the derivative $Df(\phi):C^{1}\to\R^{n}$ extends to a linear map $D_{e}f(\phi):C\to\R^{n}$ such that
  \begin{equation*}
    U\times C\ni(\phi,\psi)\mapsto D_{e}f(\phi)\psi\in\R^{n}
  \end{equation*}
  is continuous.
\end{itemize}
In particular, these conditions are typically satisfied in cases where $f$ represents the right-hand side of a differential equation with state-dependent delay.  Provided that $f$ satisfies (S1) and (S2), the results in Walther \cite{Walther2003} show that for each $\phi\in X_{f}$ with $X_{f}$ defined by
\begin{equation*}
  X_{f}:=\left\lbrace\psi\in U\mid \psi^{\prime}(0)=f(\psi)\right\rbrace,
\end{equation*}
there is a uniquely determined $t_{+}(\phi)>0$ and a (in the forward $t$-direction) non-continuable solution $x^{\phi}:[-h,t_{+}(\phi))\to\R^{n}$ of Eq. \eqref{eq: FDE} with initial value $x^{\phi}_{0}=\phi$. Moreover, all segments $x_{t}^{\phi}$, $0\leq t< t_{+}(\phi)$ and $\phi\in X_{f}$, are contained in the \textit{solution manifold} $X_{f}$ and the relations
\begin{equation*}
  F(t,\phi):=x^{\phi}_{t}
\end{equation*}
define a continuous semiflow $F:\Omega\to X_{f}$ with domain
\begin{equation*}
  \Omega:=\left\lbrace(t,\psi)\in[0,\infty)\times X_{f}\mid 0\leq t< t_{+}(\psi)\right\rbrace
\end{equation*}
and continuously differentiable time-$t$-maps
\begin{equation*}
  F_{t}:\lbrace \psi\in X_{f}\mid 0\leq t<t_{+}(\psi)\rbrace\ni\phi\mapsto F(t,\phi)\in X_{f}.
\end{equation*}

In the context of the semiflow $F$, the trivial solution $x:\R\ni t\mapsto 0\in\R^{n}$ of Eq. \eqref{eq: FDE} is the equivalent of the stationary point $\phi_{0}:=0\in X_{f}$ of $F$ as we clearly have $F(t,0)=0$ for all $t\in\R$. In order to describe the qualitative behavior of some other solutions of Eq. \eqref{eq: FDE} in close vicinity of the trivial one, it is natural to analyse the stability properties of the stationary point $\phi_{0}$ of $F$. Recall that $\phi_{0}$ is
called \textit{stable} if for each $\epsilon>0$ there is some constant $\delta(\epsilon)>0$ such that for all $\phi\in X_{f}$ with $\|\phi-\phi_{0}\|_{C^{1}}=\|\phi\|_{C^{1}}<\delta(\epsilon)$ it follows that $t_{+}(\phi)=\infty$ and that $\|F(t,\phi)-F(t,\phi_{0})\|_{C^{1}}=\|F(t,\phi)\|_{C^{1}}<\epsilon$ for all $t\geq 0$. Otherwise, we call $\phi_{0}$ \textit{unstable}. So, in the situation of stability of $\phi_{0}$, each sufficiently small initial value $\phi\in X_{f}$ leads to a solution $x^{\phi}$ of Eq. \eqref{eq: FDE} which exists and remains small for all non-negative $t\in\R$. On the other hand, if $\phi_{0}$ is unstable then there exists an open neighborhood $V$ of $0\in X_{f}$ with the property that for any $\delta>0$ we find an initial value $\phi\in V$ with $\|\phi-\phi_{0}\|_{C^{1}}=\|\phi\|_{C^{1}}<\delta$ such that the associated trajectory $[0,t_{+}(\phi))\ni t\mapsto x_{t}^{\phi}=F(t,\phi)\in X_{f}$ of $F$ leaves the neighborhood $V$ of $\phi_{0}=0$ for some finite $0<t<t_{+}(\phi)$.

One of the most common methods for the stability analysis of stationary points of flows or semiflows is based on the study of the linearization and its spectrum. In the situation of the semiflow $F$ and the stationary point $\phi_{0}$ considered here, the linearization is given by the $C_{0}$-semigroup $T:=\lbrace T(t)\rbrace_{t\geq 0}$ of bounded linear operators $T(t):=D_{2}F(t,\phi_{0})=D_{2}F(t,0)$ acting on the Banach space
\begin{equation*}
  T_{0}X_{f}:=\lbrace\psi\in C^{1}\mid \psi^{\prime}(0)=Df(0)\psi\rbrace,
\end{equation*}
which is equipped with the norm $\|\cdot\|_{C^{1}}$ of the larger space $C^{1}$. For the action of an operator $T(t)$ on some $\chi\in T_{0}X_{f}$ we have $T(t)\chi=v_{t}^{\chi}$, where $v^{\chi}_{t}$ is the segment of the unique solution $v^{\chi}:[-h,\infty)\to\R^{n}$ of the linearized variational equation
\begin{equation}\label{eq: linearization}
  v^{\prime}(t)=Df(0)v_{t}
\end{equation}
with initial value $v_{0}^{\chi}=\chi$. In particular, $0\in T_{0}X_{f}$ and $T(t)0=0$ for all $t\geq 0$; that is, $0\in T_{0}X_{f}$ is a stationary point of $T$ and, in context of Eq. \eqref{eq: linearization}, it is the equivalent of the trivial solution $v:\R\ni t\mapsto 0\in \R^{n}$. The infinitesimal generator of $T$ is the linear operator $G:\mathcal{D}(G)\ni\chi\mapsto \chi^{\prime}\in T_{0}X_{f}$
with domain
\begin{equation*}
  \mathcal{D}(G):=\lbrace\psi\in C^{2}\mid \psi^{\prime}(0)=Df(0)\psi, \psi^{\prime\prime}(0)=Df(0)\psi^{\prime}\rbrace
\end{equation*}
contained in the set $C^{2}$ of all twice continuously differentiable $\chi:[-h,0]\to\R^{n}$.

Now, the spectrum $\sigma(G)\subset \C$ of $G$ determines not only the stability properties of the trivial stationary point of the linearization $T$ of $F$ but in certain situations also the stability properties of the trivial stationary point $\phi_{0}=0$ of $F$. To make it more apparent, observe that by using the linear operator $L:=Df(0)\in\mathcal{L}(C^{1},\R^{n})$ and the generally nonlinear map $g:U\ni\phi\mapsto f(\phi)-L\phi\in\R^{n}$, we may rewrite Eq. \eqref{eq: FDE} into the form
\begin{equation}\label{eq: FDE_linear_nonlinear}
  x^{\prime}(t)=Lx_{t}+g(x_{t}).
\end{equation}
Both, $L$ and $g$, inherit properties (S1) and (S2) from $f$, and we clearly have $g(0)=0\in\R^{n}$ and $Dg(0)=0\in\mathcal{L}(C^{1},\R^{n})$. So, in close vicinity of $0\in C^{1}$ the map $g$ is small in a sense and, under certain conditions on $\sigma(G)$, the linear part on the right-hand side of Eq. \eqref{eq: FDE_linear_nonlinear} has such a strong impact on the local dynamic near the origin such that the trivial solution of Eq. \eqref{eq: FDE_linear_nonlinear}, and so of Eq. \eqref{eq: FDE}, has the same stability properties as the trivial solution of the linearized variational equation \eqref{eq: linearization}. However, before we will discuss this point in length, we shall point out that the semigroup $T$ and its generator $G$ are closely related to another strongly continuous semigroup and the associated infinitesimal generator. For this purpose, recall that due to assumption (S2) the operator $Df(0):C^{1}\to\R^{n}$ extends to a bounded linear operator $L_{e}:=D_{e}f(0):C\to\R^{n}$. In particular, $L_{e}$ defines the linear retarded functional differential equation
\begin{equation*}
  v^{\prime}(t)=L_{e}v_{t}.
\end{equation*}
The corresponding Cauchy problem
\begin{equation*}
  \left\lbrace
  \begin{aligned}
    v^{\prime}(t)&=L_{e}v_{t}\\
    v_{0}&=\chi
  \end{aligned}
  \right.
\end{equation*}
 has for each $\chi\in C$ a uniquely determined solution; that is, for each $\chi\in C$ there is a unique continuous function $v^{\chi}:[-h,\infty)\to\R^{n}$ which is continuously differentiable on $(0,\infty)$, satisfies the linear retarded functional differential equation for all $t>0$, and the segment of $v^{\chi}$ at $t=0$ coincides with the initial value $\chi$. The relations $T_{e}(t)\chi=v_{t}^{\chi}$ for $\chi\in C$ and $t\geq 0$ induce a $C_{0}$-semigroup $T_{e}:=\lbrace T_{e}(t)\rbrace_{t\geq 0}$ on the Banach space $C$. Its infinitesimal generator is given by $G_{e}:\mathcal{D}(G_{e})\ni \chi\mapsto \chi^{\prime}\in C$ with domain
 \begin{equation*}
   \mathcal{D}(G_{e}):=\lbrace \psi\in C^{1}\mid\psi^{\prime}(0)=L_{e}\psi\rbrace.
 \end{equation*}
 The last set clearly coincides with $T_{0}X_{f}$. Moreover, as discussed in Hartung et al. \cite{Hartung2006}, we have $T(t)\phi=T_{e}(t)\phi$ for all $\phi\in \mathcal{D}(G_{e})$ and all $t\geq 0$, and the two spectra $\sigma(G_{e}),\sigma(G)\subset\C$ of the generators $G_{e}$, $G$, respectively, are identical. The spectrum $\sigma(G_{e})$, and so as well the spectrum $\sigma(G)$, is given by the zeros of a familiar characteristic equation. It is discrete and contains only eigenvalues whose generalized eigenspaces are finite-dimensional. In addition, for any $\beta\in\R$ the intersection $\lbrace\lambda\in \C\mid\Re(\lambda)>\beta\rbrace\cap \sigma(G_{e})$ is either empty or finite.

But let us return to the question of stability of the stationary point $\phi_{0}=0$ of the semiflow $F$.
Suppose that spectrum $\sigma(G_{e})$ and so $\sigma(G)$ contains at least one eigenvalue with positive real part. Then, by the \textit{principle of linearized instability}, $\phi_{0}$ is an unstable stationary point of the semiflow $F$. On the other hand, assume that all eigenvalues of $G_{e}$ have negative real part. Then, by the \textit{principle of linearized stability}, $\phi_{0}$ is a stable stationary point of the semiflow $F$. To be more precisely, in this situation $\phi_{0}$ is even \emph{locally asymptotically stable}; that is, it is stable and \emph{attractive}. Here, the last point means that there is some $\epsilon>0$ such that for all $\phi\in X_{f}$ with $\|\phi-\phi_{0}\|_{C^{1}}=\|\phi\|_{C^{1}}<\epsilon$ we have $t_{+}(\phi)=\infty$ and
\begin{equation*}
  \|F(t,\phi)-\phi_{0}\|_{C^{1}}=\|F(t,\phi)\|_{C^{1}}=\|x^{\phi}_{t}\|_{C^{1}}\to 0\qquad\text{as}\quad t\to\infty.
\end{equation*}
So, each sufficiently small initial data in $X_{f}$ does lead to a solution of Eq. \eqref{eq: FDE} that does not only exist and stay small for all $t\geq 0$ but also converges to $0$ as $t\to\infty$. In summary, we obtain the following theorem about local stability analysis of the semiflow $F$ at the stationary point $\phi_{0}=0$ via the spectrum of its linearization.
\begin{thm}\label{thm: linearized}
  Let $f:U\to\R^{n}$ defined on some open neighborhood $U\subset C^{1}$ of $0\in C^{1}$ with $f(0)=0$ be given and suppose that $f$ satisfies the two smoothness assumptions (S1) and (S2).
  \begin{itemize}
    \item[(i)] If there is some $\lambda\in \sigma(G_{e})$ with $\Re(\lambda)>0$ then $\phi_{0}=0$ is unstable for the semiflow $F$ generated by Eq. \eqref{eq: FDE}.
    \item[(ii)] If $\Re(\lambda)<0$ for all $\lambda\in\sigma(G_{e})$ then $\phi_{0}=0$ is locally asymptotically stable for the semiflow $F$ generated by Eq. \eqref{eq: FDE}.
  \end{itemize}
\end{thm}

A detailed proof of assertion (ii) is contained in Hartung et al. \cite[Theorem 3.6.1]{Hartung2006} whereas a proof of statement (i) can be found in \cite[Proposition 1.4]{Stumpf2010}.
Further, we shall mentioned two points related to the theorem above.
\begin{remark}
  1. Recall, that in general, the properties of stability and attraction are independent from one another. In particular, there exist both examples of stable but not attractive as well as examples of attractive but not stable stationary points of semiflows (compare, for instance, Amann \cite[Remark 15.1(d)]{Amann1990}).

  2. The assertion of the principle of the linearized stability, that is, part (ii) of Theorem \ref{thm: linearized}, goes even further than only local asymptotic stability of the stationary point $\phi_{0}$. In fact, the rate of the attraction is exponential. More precisely, we find reals $\epsilon>0$, $K>0$ and $\omega>0$ such that for all $\phi\in X_{f}$ with $\|\phi-\phi_{0}\|_{C^{1}}=\|\phi\|_{C^{1}}<\epsilon$ we have $t_{+}(\phi)=\infty$ and
\begin{equation*}
  \|F(t,\phi)-\phi_{0}\|_{C^{1}}=\|F(t,\phi)\|_{C^{1}}\leq K\,e^{-\omega t}
\end{equation*}
as $t\geq 0$.
\end{remark}

The new ingredient of this paper is now the study of the situation where, under the standing smoothness assumptions (S1) and (S2) on $f$, an application of Theorem \ref{thm: linearized} fails in order to draw any conclusions about the local stability properties of $\phi_{0}$ from the linearized differential equation and its spectrum. This clearly occurs when the spectrum $\sigma(G_{e})$ of the linearization does not have any eigenvalue with positive real part but at least one eigenvalue on the imaginary axis. In our main result Theorem \ref{thm: reduction} we show that, under the described conditions, $\phi_{0}$ has the same local stability behavior as the zero solution of the ordinary differential equation obtained by the reduction of Eq. \eqref{eq: FDE} to a local center manifold of $F$ at $\phi_{0}$.

Note that Theorem \ref{thm: reduction} is completely in analogy with the theory of ordinary differential equations where the analog statement is known as the \textit{Pliss reduction principle} (compare Pliss \cite{Pliss1964} and Vanderbauwhede \cite[Theorem 5.18]{Vanderbauwhede1989}). Moreover, in order to show Theorem \ref{thm: reduction} we follow the proof of the Pliss reduction principle given in Vanderbauwhede \cite{Vanderbauwhede1989} and at the first glance we will need only negligible modifications. But observe that the key ingredient of the approach is an attraction property of so-called local center-unstable manifolds, and the proof of this attraction property in case of Eq. \eqref{eq: FDE} differs in some parts essentially from the one in the situation of an ordinary differential equation. A reason for that is the fact that, in contrast to an ordinary differential equation, a solution of Eq. \eqref{eq: FDE} may generally not be continued in the backward time direction. However, the attraction property used in this paper is stated in Proposition \ref{prop: attraction} and it is a consequence of the main results in \cite{Stumpf2014}.

The rest of this paper is organized as follows. The next section is devoted to local invariant manifolds of the semiflow $F$ at the stationary point $\phi_{0}=0$. At first, we recap the existence and some properties of so-called local center-unstable and local center manifolds.
After that we show, that under certain assumption on the spectrum $\sigma(G_{e})$, the two classes of local invariant manifolds coincide in the sense that each local center-unstable manifold is also a local center manifold and vice versa. Then we proceed with the discussion of the possibility to reduce the dynamic of the semiflow $F$ near $\phi_{0}=0$ to such an invariant manifold. This point will be essential for the formulation of our main result, which we will state and prove in Section 3. In the final section, we close the present paper with the discussion of a concrete example for the application of Theorem \ref{thm: linearized} as well as of Theorem \ref{thm: reduction}.

\section{Local center and center-unstable manifolds}
Recall that the spectrum $\sigma(G_{e})$ is given by the zeros of a characteristic equation, it is discrete and it consists only of eigenvalue with finite dimensional generalized eigenspaces. In addition, we have the decomposition
\begin{equation*}
\sigma(G_{e})=\sigma_{u}(G_{e})\cup \sigma_{c}(G_{e})\cup \sigma_{s}(G_{e})
\end{equation*}
where $\sigma_{u}(G_{e})$, $\sigma_{c}(G_{e})$ and $\sigma_{s}(G_{e})$ are subsets of $\sigma(G_{e})$ with eigenvalues with positive, zero, and negative real part, respectively.
Since for every $\beta\in\R$ the set $\lbrace\lambda\in\C\mid\Re(\lambda)>\beta\rbrace\cap \sigma(G_{e})$ is either empty or finite, each of the sets $\sigma_{u}(G_{e})$ and $\sigma_{c}(G_{e})$ is either empty or finite as well. In particular, the associated realified generalized eigenspaces $C_{u}\subset C$ and $C_{c}\subset C$, which are called the \emph{unstable} and the \emph{center space}, respectively, are both finite dimensional and contained in $\mathcal{D}(G_{e})\subset C^{1}$. On the other hand, the \emph{stable space} $C_{s}\subset C$, which is the realified generalized eigenspace associated with $\sigma_{s}(G_{e})$, is infinite dimensional. However, each of these three subspaces is invariant under the generator $G_{e}$ and altogether they provide the decomposition
\begin{equation*}
  C=C_{u}\oplus C_{c}\oplus C_{s}
\end{equation*}
of the Banach space $C$. Moreover, the intersection $C^{1}_{s}:=C_{s}\cap C$ is closed in $C^{1}$ such that we also obtain the decomposition
\begin{equation*}
  C^{1}=C_{u}\oplus C_{c}\oplus C_{s}^{1}
\end{equation*}
of the smaller Banach space $C^{1}$.

Assume now that, apart from our assumptions on $f$ so far, $C_{cu}:=C_{c}\oplus C_{u}\not=\lbrace 0\rbrace$. Then the main results in \cite[compare Theorems 1 \& 2]{Stumpf2011} show that in close vicinity of the origin in $X_{f}$ we find a so-called \textit{local center-unstable manifold} $W_{cu}$ of $F$ at the stationary point $\phi_{0}=0$; that is, there exist open neighborhoods $C_{cu,0}$ of $0$ in $C_{cu}$ and $C_{s,0}^{1}$ of $0$ in $C^{1}_{s}$ with $N_{cu}:=C_{cu,0}+ C_{s,0}^{1}\subset U$ and a continuously differentiable map $w_{cu}:C_{cu,0}\to C_{s,0}^{1}$ with $w_{cu}(0)=0$ and $Dw_{cu}(0)=0$ such that the graph
\begin{equation*}
  W_{cu}:=\lbrace\phi+w_{cu}(\phi)\mid\phi\in C_{cu,0}\rbrace,
\end{equation*}
which clearly contains $\phi_{0}=0$, has the properties below.
\begin{itemize}
    \item[(i)] $W_{cu}\subset X_{f}$ and $W_{cu}$ is a $C^{1}$-submanifold of $X_{f}$ with $\dim W_{cu}=\dim C_{cu}$.
    \item[(ii)] $W_{cu}$ is positively invariant with respect the semiflow $F$ relative to $N_{cu}$; that is, for each $\phi\in W_{cu}$ and all $t\geq 0$ with $\lbrace F(s,\phi)\mid 0\leq s\leq t\rbrace\subset N_{cu}$ we have $\lbrace F(s,\phi)\mid 0\leq s\leq t\rbrace\subset W_{cu}$.
    \item[(iii)] $W_{cu}$ contains the image $\gamma((-\infty,0])$ of any trajectory $\gamma:(-\infty,0]\to X_{f}$ of $F$ with $\gamma(t)\in N_{cu}$ for all $t\leq 0$.
\end{itemize}

As proven in \cite[Theorem 1.2]{Stumpf2014} such a local center-unstable manifold $W_{cu}$ is also attractive in the following sense:
For each $\phi\in X_{f}$ with $t_{+}(\phi)=\infty$ and with $F(t,\phi)$ being sufficiently small for all $t\geq 0$ there is some $\psi\in X_{f}$ with $t_{+}(\psi)=\infty$ such that $F(t,\psi)\in W_{cu}$ as $t\geq 0$ and such that $F(t,\psi)-F(t,\phi)\to 0$ exponentially for $t\to \infty$.
 In other words, $W_{cu}$ attracts the segments of all solutions of Eq. \eqref{eq: FDE_linear_nonlinear} which exist and remain sufficiently close to the stationary point $\phi_{0}=0$ for all $t\geq 0$. But even more is true as we shall see in the next proposition that forms a local version of Corollary 5.11 in \cite{Stumpf2014}.

\begin{prop}\label{prop: attraction}
  There exist open neighborhoods $\mathcal{V},\mathcal{D}$ of $\phi_{0}=0$ in $X_{f}$, a real $\eta_{A}>0$, and a continuous map $H:\mathcal{D}\to W_{cu}$ with $H(0)=0$ such that for each $\epsilon_{A}>0$ there is some $\delta_{A}>0$ with the property that for all $\phi\in \mathcal{D}$ with $\|\phi\|_{C^{1}}\leq \delta_{A}$ and all
  $t\in[0,t_{+}(\phi))\cap [0,t_{+}(H(\phi)))$ with $F(s,\phi),F(s,H(\phi))\in \mathcal{V}$ as $0\leq s\leq t$
  \begin{equation*}
    \|F(t,\phi)-F(t,H(\phi))\|_{C^{1}}\leq \epsilon_{A} e^{-\eta_{A} t}.
  \end{equation*}
\end{prop}
\begin{proof}
  The assertion follows by application of Corollary 5.11 in \cite{Stumpf2014} and subsequent restriction of the statement to a neighborhood of the stationary point $\phi_{0}=0$. In order to be more precisely, recall from \cite{Stumpf2014} that by construction $W_{cu}$ is the subset of a global center-unstable manifold $W^{\eta}$, $\eta>0$, that is contained in some open neighborhood, say $\mathcal{O}$, of $0$ in $U$. Moreover, the manifold $W^{\eta}$ is attractive in the sense of Theorem 4.1 in \cite{Stumpf2014} and this attraction property is formulated by making use of a continuous semiflow $F_{\delta}:[0,\infty)\times X_{\delta}\to X_{\delta}$ on a state space $X_{\delta}\subset C^{1}$ with $0\in X_{\delta}$ and of a continuous map $H^{\eta}_{cu}:X_{\delta}\to W^{\eta}$ with $H^{\eta}_{cu}(0)=0$.

  Next, observe that in each sufficiently small neighborhood of $0$ in $U$, the state space $X_{\delta}$ coincides with $X_{f}$ and each time-$t$-map $F_{\delta}(t,\cdot)$ takes the same values as $F(t,\cdot)$. Let $\mathcal{V}$ denote such a neighborhood of $0$ in $U$. Set $\mathcal{D}:=\mathcal{V}\cap X_{\delta}\subset X_{f}$ and $H:=H^{\eta}_{cu}\vert_{\mathcal{D}}$ . Then $H$ is clearly continuous and satisfies $H(0)=0$. Further, we claim that we also may assume that $H(\mathcal{D})\subset W_{cu}$. Indeed, in other case we could choose $\mathcal{D}\cap H^{-1}(W^{\eta}\cap \mathcal{O})=\mathcal{D}\cap H^{-1}(W_{cu})$ as the new domain of the map $H$. Now, given $\epsilon_{A}>0$, by Corollary 5.11 in \cite{Stumpf2014} we find some $\delta_{A}>0$ such that
  \begin{equation*}
  \|F_{\delta}(t,\phi)-F_{\delta}(t,H_{cu}^{\eta}(\phi))\|_{C^{1}}\leq \epsilon_{A} e^{-\eta t}
  \end{equation*}
  for all $t\geq 0$ and all $\phi\in X_{\delta}$ with $\|\phi\|_{C^{1}}<\delta_{A}$. Consider now any $\phi\in\mathcal{D}$ satisfying $\|\phi\|_{C^{1}}<\delta_{A}$ and suppose that for $0\leq t< \min\lbrace t_{+}(\phi),t_{+}(H(\phi))\rbrace$ we have $F(s,\phi),F(s,H(\phi))\in \mathcal{V}$ as $0\leq s\leq t$. Then $H(\phi)=H^{\eta}_{cu}(\phi)$. Moreover, $F(s,\phi)=F_{\delta}(s,\phi)$ and $F(s,H(\phi))=F_{\delta}(s,H(\phi))=F_{\delta}(s,H^{\eta}_{cu}(\phi))$ for all $0\leq s\leq t$. Hence,
  \begin{equation*}
    \begin{aligned}
      \|F(t,\phi)-F(t,H(\phi))\|_{C^{1}}=\|F_{\delta}(t,\phi)-F_{\delta}(t,H_{cu}^{\eta}(\phi))\|_{C^{1}}\leq
      \epsilon_{A}e^{-\eta t}
    \end{aligned}
  \end{equation*}
  and this proves the assertion with the choice $\eta_{A}=\eta$.
\end{proof}

The proposition above will be essential for the proof of our main result although the last one actually will concern the dynamic of the semiflow $F$ induced on a so-called \textit{local center manifold}. In order to clarify this point in some detail, suppose that, in addition to the hypothesis on $f$, $\dim C_{c}\geq 1$, that is, $C_{c}\not=\lbrace 0\rbrace$, holds. Then the results in Hartung et al. \cite[Theorem 4.1.1]{Hartung2006} and Krisztin \cite{Krisztin2006} show that we find open neighboorhoods $C_{c,0}$ of $0$ in $C_{c}$ and $C_{su,0}^{1}$ of $0$ in $C_{s}^{1}\oplus C_{u}$ with $N_{c}=C_{c,0}+ C_{su,0}^{1}\subset U$, and a $C^{1}$-smooth map $w_{c}:C_{c,0}\to C_{su}^{1}$ with $w_{c}(0)=0$ and $Dw_{c}(0)=0$ such that for the set
     \begin{equation*}
       W_{c}:=\lbrace \phi+w_{c}(\phi)\mid\phi\in C_{c,0}\rbrace,
     \end{equation*}
     which contains $\phi_{0}=0$ and is called a \textit{local center manifold} of $F$, the following holds:
     \begin{itemize}
       \item[(i)] $W_{c}\subset X_{f}$, and $W_{c}$ is a $C^{1}$-submanifold of $X_{f}$ with $\dim W_{c}=\dim C_{c}$.
       \item[(ii)] $W_{c}$ is locally positively invariant with respect to $F$ relative to $N_{c}$.
       \item[(iii)] $W_{c}$ contains the image $\gamma(\R)$ of any globally defined trajectory $\gamma:\R\to X_{f}$ of $F$ with $\gamma(t)\in N_{c}$ for all $t\in \R$.
     \end{itemize}
So, in particular, $W_{c}$ contains the segments of all globally defined and sufficiently small solutions of Eq. \eqref{eq: FDE_linear_nonlinear}, and so of Eq. \eqref{eq: FDE}. In addition, observe that we also may assume that the derivative of the map $w_{c}$ is bounded on its domain. That is a simple consequence of the $C^{1}$-smoothness of $w_{c}$ in combination with the equations $w_{c}(0)=0$ and $Dw_{c}(0)=0$ as shown below.
\begin{cor}\label{cor: bounded}
   There is no restriction of generality in assuming that for the map $w_{c}:C_{c,0}\to C_{su}^{1}$, whose graph defines the local center manifold $W_{c}$,
  \begin{equation*}
  \sup_{\phi\in C_{c,0}}\|Dw_{c}(\phi)\|<\infty
  \end{equation*}
  holds.
\end{cor}
\begin{proof}
   Assuming that $\sup_{\phi\in C_{c,0}}\|Dw_{c}(\phi)\|<\infty$ does not hold, below we construct another local center manifold $\tilde{W}_{c}$ with the desired property. For this purpose, choose some $\delta>0$ such that for the open ball $B_{\delta}(0)$ of radius $\delta$ about $0$ in the center space $C_{c}$ we have $\overline{B_{\delta}(0)}\subset C_{c,0}$. Set $\tilde{C}_{c,0}:=B_{\delta}(0)$, $\tilde{N}:=B_{\delta}(0)+C_{su,0}^{1}$ and $\tilde{w}_{c}:=w_{c}\vert_{B_{\delta}}:B_{\delta}(0)\to C_{su,0}^{1}$. Of course, $\tilde{N}\subset U$ and $\tilde{w}_{c}$ is $C^{1}$-smooth and does satisfy $\tilde{w}_{c}(0)=0$ and $D\tilde{w}_{c}(0)=0$. Furthermore, straightforward arguments show that the graph
  \begin{equation*}
    \tilde{W}_{c}:=\lbrace \phi+\tilde{w}_{c}(\phi)\mid \phi\in \tilde{C}_{c,0}\rbrace
  \end{equation*}
  of $\tilde{w}_{c}$ has the properties (i) -- (iii) of a local center manifold, whereas the set $N$ has to be replaced by $\tilde{N}$. Now, the function $\phi\mapsto \|Dw_{c}(\phi)\|$ is clearly continuous. Hence, it takes a maximum $0\leq M<\infty$ on the compact subset $\overline{B_{\delta}(0)}$ of $C_{c}$. It follows that
  \begin{equation*}
    \sup_{\phi\in \tilde{C}_{c,0}}\|D\tilde{w}_{c}(\phi)\|=\sup_{\phi\in B_{\delta}(0)}\|D w_{c}(\phi)\|
    \leq \sup_{\phi\in \overline{B_{\delta}(0)}}\|Dw_{c}(0)\|\leq M<\infty,
  \end{equation*}
  and this finally completes the proof.
\end{proof}

If now the spectral part $\sigma_{u}(G_{e})$ is not empty then, in view of the statements about the dimension, it is clear that the local manifolds $W_{cu}$ and $W_{c}$ differ from each other. But in our main result we will treat the situation where the linearization does not have any unstable direction, that is, where $\sigma_{u}=\emptyset$. And in this case, it may be assumed that both $W_{c}$ and $W_{cu}$ coincide as discussed below.
\begin{prop}
  Suppose that, apart from the smoothness assumptions (S1) and (S2) on $f$, $\sigma_{u}(G_{e})=\emptyset$ but $\sigma_{c}(G_{e})\not=\emptyset$. Then each (sufficiently small) local center-unstable manifold of $F$ at $\phi_{0}=0$ (in the sense of \cite[Theorem 1]{Stumpf2011}) is also a local center manifold of $F$ at $\phi_{0}=0$ (in the sense of \cite[Thorem 4.1.1]{Hartung2006}), and vice versa.
\end{prop}
\begin{proof}
  For each $\eta>0$ let $C^{1}_{\eta,\R}$ and $C^{1}_{\eta,(-\infty,0]}$ denote the Banach spaces
  \begin{equation*}
  C^{1}_{\eta,\R}:=\lbrace u\in C(\R,C^{1})\mid \sup_{t\in\R}e^{-\eta |t|}\|u(t)\|_{\R^{n}}<\infty\rbrace
  \end{equation*}
  and
  \begin{equation*}
  C^{1}_{\eta,(-\infty,0]}:=\lbrace u\in C((-\infty,0],C^{1})\mid \sup_{t\leq 0}e^{\eta t}\|u(t)\|_{\R}<\infty\rbrace,
  \end{equation*}
  respectively. Then recall that in the main the construction of local center manifolds in \cite{Hartung2006} runs as follows. After fixing appropriate real $\eta>0$ and small enough $\delta>0$, one considers a specific parameter-dependent contraction $\mathcal{G}^{\eta,\delta}_{c}:C^{1}_{\eta,\R}\times C_{c}\to C^{1}_{\eta,\R}$ with $\mathcal{G}^{\eta,\delta}_{c}(0,0)=0$ such that for each $\phi\in C_{c}$ the equation
  $u=\mathcal{G}^{\eta,\delta}_{c}(u,\phi)$
  has a uniquely determined solution $u(\phi)\in C^{1}_{\eta,\R}$. This leads to a continuous map $u^{\eta,\delta}_{c}:C_{c}\ni \phi\mapsto u(\phi)\in C_{\eta}^{1}$, and a local center manifold of $F$ at $\phi_{0}=0$ is then defined as that subset of $W^{\eta,\delta}_{c}:=\lbrace u^{\eta,\delta}_{c}(\phi)(0)\mid \phi\in C_{c}\rbrace$ where the parameters $\phi\in C_{c}$ are contained in the open ball of radius $\delta$ about $0$ in $C_{c}$.

  Similarly, the local center-unstable manifolds in \cite{Stumpf2011} are constructed by conside\-ring, for the same $\eta>0$ as in the case of local center manifolds and sufficiently small $\delta>0$, a parameter-dependent contraction $\mathcal{G}^{\eta,\delta}_{cu}:C^{1}_{\eta,(-\infty,0]}\times C_{cu}\to C^{1}_{\eta,(\infty,0]}$ satisfying $\mathcal{G}^{\eta,\delta}_{cu}(0,0)=0$ and leading for each fixed $\phi\in C_{cu}$ to a uniquely determined solution $u(\phi)$ of the equation $u=\mathcal{G}^{\eta,\delta}_{cu}(u,\phi)$. This results in a continuous mapping $u^{\eta,\delta}_{cu}:C_{cu}\ni\phi\mapsto u(\phi)\in C^{1}_{\eta,(-\infty,0]}$, and the restriction of $W^{\eta,\delta}_{cu}:=\lbrace u^{\eta}_{cu}(\phi)(0)\mid \phi\in C_{cu}\rbrace$ to parameters $\phi\in C_{cu}$ in the open ball of radius $\delta$ about $0$ in $C_{cu}$ defines a local center-unstable manifold of $F$ at $\phi_{0}=0$.

  Now, observe that by assumptions we have $C_{u}=\lbrace 0\rbrace\subset C^{1}$ and so $C_{c}=C_{cu}$. Therefore, a careful comparison of the definition of $\mathcal{G}^{\eta,\delta}_{c}$ in \cite{Hartung2006} and the one of $\mathcal{G}^{\eta,\delta}_{cu}$ in \cite{Stumpf2011} leads to the conclusion that for all sufficiently small $\delta>0$ we have
  \begin{equation*}
    u^{\eta,\delta}_{c}(\phi)(t)=u^{\eta,\delta}_{cu}(\phi)(t)
  \end{equation*}
  for all $\phi\in C_{c}=C_{cu}$ and all $t\leq 0$. In particular,  $u^{\eta,\delta}_{c}(\phi)(0)=u^{\eta,\delta}_{cu}(\phi)(0)$ for each $\phi\in C_{c}$. It follows that $W^{\eta,\delta}_{c}=W^{\eta,\delta}_{cu}$, which implies the assertion.
\end{proof}

From now on and until the end of the next section, we assume that the assumptions of the last result hold and we set $W_{c}=W_{cu}$, $w_{c}=w_{cu}$, and $d:=\dim W_{c}>0$. Further, let $P_{c}:C^{1}\to C_{c}$ denote the continuous projection of $C^{1}$ along $C_{s}^{1}$ onto $C_{c}=C_{cu}$.

Our next goal is to derive an ordinary differential equation describing the dyna\-mics on $W_{c}$ induced by solutions of Eq. \eqref{eq: FDE_linear_nonlinear}. For this purpose, choose a basis $\lbrace \phi_{1},\dots,\phi_{d}\rbrace\subset C^{1}$ of the center space $C_{c}$ and introduce the row vector
\begin{equation*}
  \varPhi_{c}:=(\phi_{1},\dots,\phi_{d}).
\end{equation*}
Then each $\phi\in C_{c}$ has clearly a uniquely determined representation as
\begin{equation*}
\phi=\varPhi_{c}\,c(\phi)=\sum_{j=1}^{d}\phi_{j}\,c_{j}(\phi)
\end{equation*}
with a column vector $c(\phi):=(c_{1}(\phi),\cdots,c_{d}(\phi))^{T}\in\R^{d}$. Thus, using the notation $\varGamma_{c}:C_{c}\to\R^{d}$ for the bounded linear map assigning each $\phi\in C_{c}$ the coefficient vector $c(\phi)\in\R^{d}$, we get $\phi=\varPhi_{c}\,\varGamma_{c}(\phi)$ for all $\phi\in C_{c}$.

Next, observe that, in consideration of the invariance of $C_{c}$ under $G_{e}$, we find some matrix $B_{c}\in\R^{d\times d}$ such that
\begin{equation*}
  G_{e}\,\varPhi_{c}=\varPhi_{c}\,B_{c}
\end{equation*}
with the row vector $G_{e}\,\varPhi_{c}:=(G_{e}\phi_{1},\dots,G_{e}\phi_{d})$. The eigenvalues of the matrix $B_{c}$ coincide with $\sigma_{c}(G_{e})$, that is, $B_{c}$ has the same eigenvalues as the restriction of $G_{e}$ to $C_{c}$.

Now, as discussed in \cite[Chapter 2.6]{Stumpf2010}, we find an open neighborhood $V\subset\R^{d}$ of $0\in\R^{d}$ and a continuously differentiable function $h:V\to\R^{d}$ with $h(0)=0$ and $Dh(0)=0$ such that the center manifold reduction of $F$ to $W_{c}$ reads
\begin{equation}\label{eq: CM-reduction}
  z^{\prime}(t)=B_{c}\,z(t)+h(z(t)).
\end{equation}
In other words, there are $V$ and $h$ as described above such that
on the one hand, given any solution $x:I+[-h,0]\to\R^{n}$, $I\subset\R$ an interval, of Eq. \eqref{eq: FDE_linear_nonlinear} with $x_{t}\in W_{c}$ for all $t\in I$, the function $z:I\ni t\mapsto \varGamma_{c}\, P_{c}\, x_{t}\in\R^{d}$ forms a solution of Eq. \eqref{eq: CM-reduction}. And on the other hand, for any solution $z:I\to\R^{d}$, $I\subset\R$ an interval, of Eq. \eqref{eq: CM-reduction} we find a solution $x:I+[-h,0]\to\R^{n}$ of Eq. \eqref{eq: FDE_linear_nonlinear} with $x_{t}=\varPhi_{c}\,z(t)+w_{c}(\varPhi_{c}\,z(t))$ for all $t\in I$.
\section{The reduction principle and its proof}
After all the preparatory work we are now in the position to state our main result.
\begin{thm}\label{thm: reduction}
  Let $f$ be as in Theorem \ref{thm: linearized} and suppose that $\sigma_{u}(G_{e})=\emptyset$ but $\sigma_{c}(G_{e})\not=\emptyset$.
  If $z:\R\ni t\mapsto 0\in\R^{d}$ is unstable / stable / locally asymptotically stable as a solution of Eq. \eqref{eq: CM-reduction}, then $\phi_{0}=0$ is unstable / stable / locally asymptotically stable as a stationary point of the semiflow $F$.
\end{thm}
\begin{remark}
  As already mentioned in the introduction, the above result is completely similar to the so-called \textit{Pliss reduction principle} from the theory of ordinary differential equations. For more details, we refer the reader to Pliss \cite{Pliss1964} and Vanderbauwhede \cite{Vanderbauwhede1989}.
\end{remark}
The statement of Theorem \ref{thm: linearized} consists of three parts and we show them in a series of propositions. But before doing so, we prove the following auxiliary result which is of similar type as Theorem 1.6 in Getto \& Waurick \cite{Getto2014}:

\begin{prop}\label{prop: global solutions}
  Given $f$ as in Theorem \ref{thm: linearized}, there is some $b>0$ such that if $\phi\in X_{f}$ and if the associated solution $x^{\phi}:[-h,t_{+}(\phi))\to\R^{n}$ of Eq. \eqref{eq: FDE} is bounded by $b$, that is, $\|x^{\phi}_{t}\|_{C^{1}}\leq b$ for all $0\leq t<t_{+}(\phi)$, then $t_{+}(\phi)=\infty$.
\end{prop}
\begin{proof}
   To begin with, recall from Corollary 1 in Walther \cite{Walther2004} that there is some open neighborhood $U_{B}$ about $0$ in $U$ on which the map $D_{e}f:U\to\mathcal{L}(C,\R^{n})$ is bounded. Next, note that by the continuity of the semiflow $F$ we also find some open neighborhood $U_{h}$ of $0$ in $U$ such that for all $\phi\in U_{h}\cap X_{f}$ we have $t_{+}(\phi)>h$. Fix now any $b>0$ with $\lbrace \phi\in U|\|\phi\|_{C^{1}}\leq b\rbrace\subset U_{B}\cap U_{h}$, and suppose that for given $\psi\in X_{f}$ we have $\|x_{t}^{\psi}\|_{C^{1}}=\|F(t,\psi)\|_{C^{1}}\leq b$ as $0\leq t<t_{+}(\psi)$. Then the assertion follows from Proposition 2.2 in Diekmann et al. \cite[Chapter VII.2]{Diekmann1995} provided the closure of $\lbrace x^{\psi}_{t}\mid 0\leq t<t_{+}(\psi)\rbrace\subset X_{f}$ is compact. Hence, in view of the Arzel\a'a-Ascoli Theorem, it suffices only to prove that both sets
  \begin{equation*}
    \mathcal{O}:=\lbrace x^{\psi}_{t}\mid 0\leq t<t_{+}(\psi)\rbrace\quad\text{and}\quad
    \mathcal{O}^{\prime}:=\lbrace (x^{\psi}_{t})^{\prime}\mid 0\leq t<t_{+}(\psi)\rbrace
  \end{equation*}
  are bounded with respect to the norm of $C$ and equicontinuous. This is done in the following.

  First, recall that by definition $\|x_{t}^{\psi}\|_{C^{1}}=\|x_{t}^{\psi}\|_{C}+\|(x_{t}^{\psi})^{\prime}\|_{C}$ as $0\leq t<t_{+}(\psi)$. Therefore, the assumption $\|x_{t}^{\psi}\|_{C^{1}}\leq b$ for all $0\leq t<t_{+}(\psi)$ clearly implies the boundedness of both $\mathcal{O}$ and $\mathcal{O}^{\prime}$ with respect the $\|\cdot\|_{C}$-norm. Furthermore,
  \begin{equation*}
    x^{\psi}_{t}(s)-x^{\psi}_{t}(u)=
    x^{\psi}(t+s)-x^{\psi}(t+u)
  =\int_{t+u}^{t+s}(x^{\psi})^{\prime}(v)\,dv
  \end{equation*}
  and therefore
  \begin{equation}\label{eq: estimate3}
\|x^{\psi}_{t}(s)-x^{\psi}_{t}(u)\|_{\R^{n}}  \leq \max_{-h\leq v<t_{+}(\psi)}\|(x^{\psi})^{\prime}(v)\|_{\R^{n}}|s-u|\leq b|s-u|
  \end{equation}
  for all $0\leq t<t_{+}(\psi)$ and all $s,u\in[-h,0]$. Thus, $\mathcal{O}$ is also equicontinous.

  The only point remaining concerns the equicontinuity of $\mathcal{O}^{\prime}$. In order to see this, let $\epsilon>0$ be given. Recall from the beginning of the proof that we have $t_{+}(\psi)>h$. As $(x^{\psi})^{\prime}$ is continuous, its restriction to the interval $[-h,h]$ is uniformly continuous. For this reason, there is some $\delta_{1}>0$ such that for all $u,s\in[-h,h]$ the implication
  \begin{equation*}
  |s-u|<\delta_{1}\qquad\Longrightarrow\qquad \|(x^{\psi})^{\prime}(s)-(x^{\psi})^{\prime}(u)\|_{\R^{n}}<\epsilon
  \end{equation*}
  holds. In particular, given $0\leq t\leq h$ and reals $s,u\in[-h,0]$ with $|s-u|<\delta_{1}$, we have $\|(x^{\psi}_{t})^{\prime}(s)-(x^{\psi}_{t})^{\prime}(u)\|_{\R^{n}}<\epsilon$.

 Next, consider any fixed $t\geq h$. 
 Then, for all $s,u\in[-h,0]$, we get
\begin{equation*}
  \begin{aligned}
    \|(x^{\psi}_{t})^{\prime}(s)-(x^{\psi}_{t})^{\prime}(u)\|_{\R^{n}}&=
    \|(x^{\psi})^{\prime}(t+s)-(x^{\psi})^{\prime}(t+u)\|_{\R^{n}}\\
    &=\|f(F(t+s,\psi))-f(F(t+u,\psi))\|_{\R^{n}}\\
    &=\left\|\int_{0}^{1}Df(x_{t+u}^{\psi}+v(x_{t+s}^{\psi}-x_{t+u}^{\psi}))\,\left[x_{t+s}^{\psi}-
    x_{t+u}^{\psi}\right]\,dv\right\|_{\R^{n}}\\
    &=\left\|\int_{0}^{1}D_{e}f(x_{t+u}^{\psi}+v(x_{t+s}^{\psi}-x_{t+u}^{\psi}))\,\left[x_{t+s}^{\psi}-
    x_{t+u}^{\psi}\right]\,dv\right\|_{\R^{n}}\\
    &\leq \int_{0}^{1}\left\|D_{e}f(x_{t+u}^{\psi}+v(x_{t+s}^{\psi}-x_{t+u}^{\psi}))\,\left[x_{t+s}^{\psi}-
    x_{t+u}^{\psi}\right]\right\|_{\R^{n}}\,dv
    \end{aligned}
    \end{equation*}
and so, in view of the boundedness of $\phi\mapsto D_{e}f(\phi)$ on $U_{B}$ and estimate \eqref{eq: estimate3},
    \begin{equation*}
    \begin{aligned}
   \|(x^{\psi}_{t})^{\prime}(s)-(x^{\psi}_{t})^{\prime}(u)\|_{\R^{n}}&\leq \int_{0}^{1}\sup_{\phi\in U_{B}}\|D_{e}f(\phi)\|\|x^{\psi}_{t+s}-x^{\psi}_{t+u}\|_{C}\,dv\\
    &=\sup_{\phi\in U_{B}}\|D_{e}f(\phi)\|\max_{-h\leq v\leq 0}\|x^{\psi}_{t}(s+v)-
    x_{t}^{\phi}(u+v)\|_{\R^{n}}\\
    &\leq \sup_{\phi\in U_{B}}\|D_{e}f(\psi)\|\cdot b\cdot |s-u|.
  \end{aligned}
\end{equation*}
In particular, it follows that $\|(x^{\psi}_{t})^{\prime}(s)-(x^{\psi}_{t})^{\prime}(u)\|_{\R^{n}}<\epsilon$ provided $s,u\in[-h,0]$ satisfy $|u-s|<\delta_{2}$ with
\begin{equation*}
  \delta_{2}:=\frac{\epsilon}{\sup_{\phi\in U_{B}}\|D_{e}f(\psi)\|b}.
\end{equation*}

Now, choosing $0<\delta<\min\lbrace \delta_{1},\delta_{2}\rbrace$, we see that for all $t\geq 0$ and all $s,u\in[-h,0]$ with $|s-u|<\delta$ we have $\|(x^{\psi}_{t})^{\prime}(s)-(x^{\psi}_{t})^{\prime}(u)\|_{\R^{n}}<\epsilon$. This shows the equicontinuity of $\mathcal{O}^{\prime}$ which finally finishes the proof of the proposition.
\end{proof}

Now, we return to Theorem \ref{thm: reduction} and prove the assertion that the stability of the zero solution of the reduced differential equation implies the stability of $\phi_{0}$.
\begin{prop}\label{prop: stability}
  Consider $f$ and $\sigma(G_{e})$ as in Theorem \ref{thm: reduction} and suppose that $z:\R\ni t\mapsto 0\in\R^{d}$ is stable as a solution of Eq. \eqref{eq: CM-reduction}. Then $\phi_{0}=0\in X_{f}$ is a stable stationary point of the semiflow $F$.
\end{prop}
\begin{proof}
  1. Let $\epsilon>0$ be given. We have to find some constant $\delta>0$ with the property that for each $\psi\in X_{f}$ with $\|\psi\|_{C^{1}}<\delta$ it follows $t_{+}(\psi)=\infty$ and $\|F(t,\psi)\|_{C^{1}}<\epsilon$ for all $t\geq 0$. For this purpose, let $\|\cdot\|_{\R^{d}}$ denote any norm in $\R^{d}$ and $k\geq 0$ some real ensuring $\sum_{j=1}^{d}|v_{i}|\leq k\|v\|_{\R^{d}}$ for all $v\in\R^{d}$. Further, choose any $0<\tilde{\epsilon}\leq \epsilon$ such that for all $\phi\in X_{f}$ with $\|\phi\|_{C^{1}}<\tilde{\epsilon}$ we have $\phi\in \mathcal{V}$ where $\mathcal{V}$ is the open neighborhood of $0$ in $X_{f}$ from Proposition \ref{prop: attraction}. Observe that additionally we may assume that $\tilde{\epsilon}$ satisfies $\tilde{\epsilon}<b$ with constant $b$ introduced in Proposition \ref{prop: global solutions}.

  2.  Given $\psi\in W_{c}$, suppose that for the associated solution $x^{\psi}$ of Eq. \eqref{eq: FDE_linear_nonlinear} we have $t_{+}(\psi)=\infty$ and $x_{t}^{\psi}\in W_{c}$ for all $t\geq 0$. Then, by definition, each segment $x^{\psi}_{t}$, $t\geq 0$, may be written in the form $x^{\psi}_{t}=P_{c}\,x_{t}^{\psi}+w_{c}(P_{c}\,x^{\psi}_{t})$. Moreover, the function
  \begin{equation*}
    z^{\psi}:[0,\infty)\ni t\mapsto \varGamma_{c}\, P_{c}\, x^{\psi}_{t}\in\R^{d}
  \end{equation*}
  is a solution of Eq. \eqref{eq: CM-reduction}. Next, observe that for each $t\geq 0$ we have
  \begin{equation*}
      \begin{aligned}
      \|P_{c}\,x_{t}^{\psi}\|_{C^{1}}
      &=\|\varPhi_{c}\, \varGamma_{c}( P_{c}\,x_{t}^{\psi})\|_{C^{1}}\\
      &=\|\varPhi_{c}\, z^{\psi}(t)\|_{C^{1}}\\
      &\leq\sum_{j=1}^{d}|z_{j}^{\psi}(t)|\,\|\phi_{j}\|_{C^{1}}\\
      &\leq \max_{1\leq j\leq d}\|\phi_{j}\|_{C^{1}}
      \sum_{j=1}^{d}|z_{j}^{\psi}(t)|\\
      &\leq \max_{1\leq j\leq d}\|\phi_{j}\|_{C^{1}}\,k\,\|z^{\psi}(t)\|_{\R^{d}}
      \end{aligned}
    \end{equation*}
  and therefore
  \begin{equation}\label{eq: estimate}
    \begin{aligned}
      \|x^{\psi}_{t}\|_{C^{1}}&=\|P_{c}\,x_{t}^{\psi}+w_{c}(P_{c}\,x^{\psi}_{t})\|_{C^{1}}\\
      &\leq \|P_{c}\,x_{t}^{\psi}\|_{C^{1}}+\|w_{c}(P_{c}\,x^{\psi}_{t})\|_{C^{1}}\\
      &\leq \|P_{c}\,x_{t}^{\psi}\|_{C^{1}}+\sup_{\phi\in C_{c,0}}\|Dw_{c}(\phi)\|\|P_{c}\,x_{t}^{\psi}\|_{C^{1}}\\
      &\leq (1+\sup_{\phi\in C_{c,0}}\|Dw_{c}(\phi)\|)\,\max_{1\leq j\leq d}\|\phi_{j}\|_{C^{1}}\,k\,\|z^{\psi}(t)\|_{\R^{d}}
    \end{aligned}
  \end{equation}
  with $\sup_{\phi\in C_{c,0}}\|Dw_{c}(\phi)\|$ being assumed to be bounded due to Corollary \ref{cor: bounded}.

  Now, recall that by assumption $z:\R\ni t\mapsto 0\in\R^{d}$ is stable; that is, for any fixed $\epsilon_{R}>0$ there is some $\delta_{R}(\epsilon_{R})>0$ such that for each $\tilde{z}\in\R^{d}$ with $\|\tilde{z}\|_{\R^{d}}<\delta_{R}(\epsilon_{R})$ the solution $z(\cdot;\tilde{z})$ of Eq. \eqref{eq: CM-reduction} with $z(0;\tilde{z})=\tilde{z}$ does exist for all $0\leq t<\infty$ and satisfies $\|z(t;\tilde{z})\|_{\R^{d}}<\epsilon_{R}$ as $t\geq 0$. Set
  \begin{equation*}
    \epsilon_{R}:=\frac{\tilde{\epsilon}}{2k(1+\sup_{\phi\in C_{c,0}}\|Dw_{c}(\phi)\|)\max_{1\leq j\leq d}\|\phi_{j}\|_{C^{1}}}
  \end{equation*}
  and then fix
  \begin{equation*}
    0<\delta<\frac{\delta_{R}(\epsilon_{R})}{\|\varGamma_{c} P_{c}\|}.
  \end{equation*}
  We claim that for each $\tilde{\psi}\in W_{c}$ with $\|\tilde{\psi}\|_{C^{1}}<\delta$ we have both $t_{+}(\tilde{\psi})=\infty$ and $\|F(t,\tilde{\psi})\|=\|x^{\tilde{\psi}}_{t}\|_{C^{1}}<\epsilon/2$ as $t\geq 0$. In order to see this, set $z(\tilde{\psi}):=\varGamma_{c} P_{c}\tilde{\psi}$ and observe that
  \begin{equation*}
    \|z(\tilde{\psi})\|_{\R^{d}}=\|\varGamma_{c} P_{c}\tilde{\psi}\|_{\R^{d}}\leq \|\varGamma_{c} P_{c}\|\|\tilde{\psi}\|_{C^{1}}\leq \|\varGamma_{c}P_{c}\|\delta<\delta_{R}(\epsilon_{R}).
  \end{equation*}
  Hence, the solution $z(\cdot;z(\tilde{\psi}))$ of Eq. \eqref{eq: CM-reduction} does exist for all $t\geq 0$ and additionally satisfies $\|z(t;z(\tilde{\psi}))\|_{\R^{d}}<\epsilon_{R}$ as $t\geq 0$. But then we find a solution $x:[-h,\infty)\to\R^{n}$ of Eq. \eqref{eq: FDE_linear_nonlinear} with
  \begin{equation*}
  x_{t}=\varPhi_{c} z(t;z(\tilde{\psi}))+w_{c}(\varPhi_{c}\,z(t;z(\tilde{\psi})))\in W_{c}.
  \end{equation*}
  As
  \begin{equation*}
    x_{0}=\varPhi_{c}\, z(0;z(\tilde{\psi}))+w_{c}(\varPhi_{c}\,z(0;z(\tilde{\psi}))=\varPhi_{c}\, z(\tilde{\psi})+
    w_{c}(\varPhi_{c}\,z(\tilde{\psi}))=P_{c}\tilde{\psi}+w_{c}(P_{c}\tilde{\psi})=\tilde{\psi}
  \end{equation*}
  it follows that $x_{t}=F(t,x_{0})=F(t,\tilde{\psi})=x_{t}^{\tilde{\psi}}$ as $t\geq 0$. In particular, $x_{t}^{\tilde{\psi}}\in W_{c}$ for all $t\geq 0$ such that, in view of estimate \eqref{eq: estimate}, we finally get for each $t\geq 0$
  \begin{equation}\label{eq: estimate6}
    \begin{aligned}
      \|F(t,\tilde{\psi})\|_{C^{1}}&=\|x_{t}^{\tilde{\psi}}\|_{C^{1}}\\
      &\leq (1+\sup_{\phi\in C_{c,0}}\|Dw_{c}(\phi)\|)\,\max_{1\leq j\leq d}\|\phi_{j}\|_{C^{1}}\,k\,\|z(t;z(\tilde{\psi}))\|_{\R^{d}}\\
      &\leq (1+\sup_{\phi\in C_{c,0}}\|Dw_{c}(\phi)\|)\,\max_{1\leq j\leq d}\|\phi_{j}\|_{C^{1}}\,k\,\epsilon_{R}\\
      &\leq \frac{\tilde{\epsilon}}{2}
    \end{aligned}
  \end{equation}
  as claimed. Note that, due to the choice of $\tilde{\epsilon}>0$ in the first part, the last estimate also implies $F(t,\tilde{\psi})\in \mathcal{V}$ for all $t\geq 0$.

  3. Consider now the open neighborhoods $\mathcal{V}$, $\mathcal{D}$ of $\phi_{0}=0$ in $X_{f}$, the real $\eta_{A}>0$, and the map $H:\mathcal{D}\to W_{cu}=W_{c}$ from Proposition \ref{prop: attraction}, and choose some fixed $0<\epsilon_{1}<\min\lbrace \delta,\tilde{\epsilon}/2\rbrace$. As $\mathcal{V}, \mathcal{D}$ are open, the map $H$ continuous, and $H(0)=0$, there clearly is some real $0<\delta_{1}(\epsilon_{1})<3\tilde{\epsilon}/4$ such that for all $\phi\in X_{f}$ with $\|\phi\|_{C^{1}}<\delta_{1}(\epsilon_{1})$ we have $\phi\in \mathcal{V}\cap\mathcal{D}$ and $\|H(\phi)\|_{C^{1}}<\epsilon_{1}$. Observe that for those $\phi$ it particularly follows that $H(\phi)\in W_{c}$ and $\|H(\phi)\|_{C^{1}}<\delta$. Hence, given any $\phi\in X_{f}$ with $\|\phi\|_{C^{1}}<\delta_{1}(\epsilon_{1})$, from the last part we conclude that $t_{+}(H(\phi))=\infty$ and that for each $t\geq 0$ we have both $\|F(t,H(\phi))\|_{C^{1}}<\tilde{\epsilon}/2$ and $F(t,H(\phi))\in \mathcal{V}$.

  Next, note that we may assume that the real $\delta_{1}(\epsilon_{1})>0$ is sufficiently small such that the estimate of Proposition \ref{prop: attraction} holds with constants $\epsilon_{A}=\tilde{\epsilon}/4$ and $\delta_{A}=\delta_{1}(\epsilon_{1})$; that is, we may assume that, for each $\phi\in X_{f}$ satisfying $\|\phi\|_{C^{1}}\leq \delta_{1}(\epsilon_{1})$ and for all $0\leq t<\min\lbrace t_{+}(\phi),t_{+}(H(\phi))\rbrace$ with $F(s,\phi),F(s,H(\phi))\in \mathcal{V}$ as $0\leq s\leq t$,
  \begin{equation}\label{eq: estimate2}
    \|F(t,\phi)-F(t,H(\phi))\|_{C^{1}}\leq \frac{\tilde{\epsilon}e^{-\eta_{A}t}}{4}
  \end{equation}
  holds. But by the last part, for all $\phi\in X_{f}$ with $\|\phi\|_{C^{1}}\leq \delta_{1}(\epsilon_{1})$ we have $\phi\in \mathcal{V}$, $t_{+}(H(\phi))=\infty$ and $F(t,H(\phi))\in \mathcal{V}$ as $t\geq 0$. For this reason, given $\phi\in X_{f}$ satisfying $\|\phi\|_{C^{1}}\leq \delta_{1}(\epsilon_{1})$, from estimates \eqref{eq: estimate2} we first get
  \begin{equation}\label{eq: estimate4}
    \|F(t,\phi)\|_{C^{1}}\leq \frac{\tilde{\epsilon}e^{-\eta_{A}t}}{4}+\|F(t,H(\phi))\|_{C^{1}}
  \end{equation}
  and then by using estimate \eqref{eq: estimate6}
  \begin{equation}\label{eq: estimate5}
    \|F(t,\phi)\|_{C^{1}}\leq \frac{\tilde{\epsilon}}{4}+\frac{\tilde{\epsilon}}{2}=\frac{3\tilde{\epsilon}}{4}
  \end{equation}
  as long as $F(t,\phi)$ does exist and $[0,t]\ni s\mapsto F(s,\phi)\in X_{f}$ stays in $\mathcal{V}$.

  Now, suppose that for any $\phi\in X_{f}$ with $\|\phi\|_{C^{1}}<\delta_{1}(\epsilon_{1})$ the associated trajectory $[0,t_{+}(\phi))\ni t\mapsto F(t,\phi)\in X_{f}$ would leave the neighborhood $\mathcal{V}$ of $0\in X_{f}$ for some $0<t_{\mathcal{V}}<t_{+}(\phi)$. Then, in view of $\delta_{1}(\epsilon_{1})<3\tilde{\epsilon}/4$ and $\lbrace \psi\in X_{f}\mid \|\psi\|_{C^{1}}<\tilde{\epsilon}\rbrace\subset \mathcal{V}$, we clearly would find some $0<t_{L}<t_{\mathcal{V}}$ with $F(t,\phi)\in \mathcal{V}$ for all $0\leq t<t_{L}$ and $\|F(t_{L},\phi)\|_{C^{1}}=\tilde{\epsilon}$. Especially, in consideration of estimate \eqref{eq: estimate5}, it would follow that $\|F(t,\phi)\|_{C^{1}}\leq 3\tilde{\epsilon}/4$ for all $0\leq t<t_{L}$ but $\|F(t_{L},\phi)\|_{C^{1}}=\tilde{\epsilon}$, which is a contradiction to the continuity of the map $[0,t_{+}(\phi))\ni t\mapsto \|F(t,\phi)\|_{C^{1}}\in[0,\infty)$. Thus, for all $\phi\in X_{f}$ satisfying $\|\phi\|_{C^{1}}<\delta_{1}(\epsilon_{1})$ and all $0\leq t<t_{+}(\phi)$ we have $F(t,\phi)\in \mathcal{V}$ such that both estimates \eqref{eq: estimate4} and \eqref{eq: estimate5} are fulfilled as $0\leq t<t_{+}(\phi)$. But then it is also clear that $t_{+}(\phi)=\infty$ for each $\phi\in X_{f}$ with $\|\phi\|_{C^{1}}<\delta_{1}(\epsilon_{1})$. Indeed, it immediately follows from Proposition \ref{prop: global solutions} as $\tilde{\epsilon}<b$ by our choice of $\epsilon$ in the first part. Consequently, we see that for all $\phi\in X_{f}$ with $\|\phi\|_{C^{1}}<\delta_{1}(\epsilon_{1})$ we have $t_{+}(\phi)=\infty$ and $\|F(t,\phi)\|_{C^{1}}\leq 3\tilde{\epsilon}/4<\tilde{\epsilon}\leq \epsilon$ as $t\geq 0$. This proves the assertion.
\end{proof}

Next, we extend the arguments in the last proof in order to show that $\phi_{0}$ is locally asymptotically stable provided the zero solution of the reduced differential equation is so.
\begin{prop}\label{prop: asym_stability}
  Suppose that, under the assumptions of Proposition \ref{prop: stability}, the function $z$ is not only stable but locally asymptotically stable as a solution of \eqref{eq: CM-reduction}. Then $\phi_{0}=0\in X_{f}$ is a locally asymptotically stable stationary point of $F$.
\end{prop}
\begin{proof}
  Revisit the proof of Proposition \ref{prop: stability} and suppose that the solution $z$ is not only stable but also attractive, that is, suppose that there is some $A_{o}>0$ such that for all $\tilde{z}\in\R^{d}$ with $\|\tilde{z}\|_{\R^{d}}\leq A_{o}$ the solution $z(\cdot;\tilde{z})$ of Eq. \eqref{eq: CM-reduction} does exist for all $t\geq 0$ and converges to $0\in\R^{d}$ as $t\to\infty$. Then we may assume that $A_{o}< \delta_{R}(\epsilon_{R})$ holds, since otherwise we could take the real $\delta_{R}(\epsilon_{R})$ instead of $A_{o}$ for the definition of an attraction region of $z$. Further, combining the continuity of the map $H$ with $H(0)=0$, we find some $0<A_{d}<\delta_{1}(\epsilon_{1})$ such that for all $\phi\in X_{f}$ with $\|\phi\|_{C^{1}}<A_{d}$ it follows that $\|H(\phi)\|_{C^{1}}<A_{o}/\|\mathit{\Gamma}_{c}P_{c}\|$. Consider now any $\psi\in X_{f}$ with $\|\psi\|_{C^{1}}<A_{d}$, and set $\tilde{\psi}:=H(\psi)\in W_{cu}=W_{c}$ and $z(\tilde{\psi}):=\mathit{\Gamma}_{c}P_{c}\tilde{\psi}$. By assumption, $\|\psi\|_{C^{1}}<\delta_{1}(\epsilon_{1})$ and so $\|\tilde{\psi}\|_{C^{1}}<\epsilon_{1}<\delta$. Therefore, both
  solutions $x^{\psi}$ and $x^{\tilde{\psi}}$ of Eq. \eqref{eq: FDE_linear_nonlinear} do exist for all $t\geq 0$. Furthermore, all segments of $x^{\tilde{\psi}}$ belong to $W_{c}$ such that from the estimates \eqref{eq: estimate4} and \eqref{eq: estimate} we conclude that
  \begin{equation}\label{eq: estimate7}
  \begin{aligned}
    \|F(t,\psi)\|_{C^{1}}&\leq \frac{\tilde{\epsilon}e^{-\eta_{A}t}}{4}+\|F(t,\tilde{\psi})\|_{C^{1}}\\
    &\leq\frac{\tilde{\epsilon}e^{-\eta_{A}t}}{4}+(1+\sup_{\phi\in C_{c,0}}\|Dw_{c}(\phi)\|)\,k\,\max_{1\leq j\leq d}\|\phi_{j}\|_{C^{1}}\|z(t;z(\tilde{\psi}))\|_{\R^{d}}
  \end{aligned}
  \end{equation}
  for all $t\geq 0$. As $\eta_{A}>0$ and
  \begin{equation*}
    \|z(0;z(\tilde{\psi}))\|_{\R^{d}}=\|z(\tilde{\psi})\|_{\R^{d}}=\|\mathit{\Gamma}_{c}P_{c}\tilde{\psi}\|_{\R^{d}}\leq \|\mathit{\Gamma}_{c}P_{c}\|\|\tilde{\psi}\|_{C^{1}}< \|\mathit{\Gamma}_{c}P_{c}\|\frac{A_{o}}{\|\mathit{\Gamma}_{c}P_{c}\|}=A_{0}
  \end{equation*}
   it follows that the right-hand side of \eqref{eq: estimate7} converges to $0\in\R$ as $t\to\infty$. But then we clearly also have $F(t,\psi)\to 0$ for $t\to\infty$, and this proves the assertion.
\end{proof}

Finally, we complete the proof of Theorem \ref{thm: reduction} by showing that in the case of an unstable zero solution of the reduced differential equation the trivial stationary point $\phi_{0}$ of $F$ is unstable as well.
\begin{prop}\label{prop: instability}
  Given $f$ and $\sigma(G_{e})$ as in Theorem \ref{thm: reduction}, assume that the zero function $z(t)=0\in\R^{d}$, $t\in\R$, is unstable as a solution of Eq. \eqref{eq: CM-reduction}. Then $\phi_{0}=0\in X_{f}$ is an unstable stationary point of $F$.
\end{prop}
\begin{proof}
  Suppose that $\phi_{0}=0$ is a stable stationary point of $F$. Then we claim that the zero solution of Eq. \eqref{eq: CM-reduction} is stable as well such that the assertion of the proposition follows by the contrapositive.

  In order to see our claim, let $\epsilon>0$ be given. Then, by using the continuity of $P_{c}$ and $\id_{C^{1}}-P_{c}$, we find a real $0<\epsilon_{d}<\epsilon/\|\mathit{\Gamma}_{c}P_{c}\|$
  ensuring both
  \begin{equation*}
  \lbrace P_{c}\phi\mid \phi\in C^{1} \text{ with }\|\phi\|_{C^{1}}<\epsilon_{d}\rbrace\subset C_{c,0}
  \end{equation*}
  and
  \begin{equation*}
   \lbrace (\id_{C^{1}}-P_{c})\phi\mid\phi\in C^{1}\text{ with }\|\phi\|_{C^{1}}<\epsilon_{d}\rbrace\subset C_{su,0}^{1}.
  \end{equation*}
  Additionally, in consideration of the stability of $\phi_{0}$, there is some $0<\delta_{d}< \epsilon_{d}$ such that for all $\phi\in X_{f}$ with $\|\phi\|_{C^{1}}<\delta_{d}$ we have $t_{+}(\phi)=\infty$ and $\|F(t,\phi)\|_{C^{1}}< \epsilon_{d}$ as $0\leq t<t_{+}(\phi)$.

  Choose now
  \begin{equation*}
    0<\delta<\min\left\lbrace\frac{\delta_{d}}{(1+\sup_{\phi\in C_{c,0}}\|Dw_{c}(\phi)\|)k\max_{1\leq i\leq d}\|\phi_{i}\|_{C^{1}}},\delta_{d}\right\rbrace
  \end{equation*}
  and consider any $\tilde{z}\in\R^{d}$ satisfying $\|\tilde{z}\|_{\R^{d}}<\delta$. We show that the solution $z(\cdot,\tilde{z})$ of Eq. \eqref{eq: CM-reduction} with $z(0;\tilde{z})=\tilde{z}$ does exist and is bounded by $\epsilon$ for all non-negative reals. For this purpose, observe that for $\tilde{\psi}:=\varPhi_{c}\tilde{z}+w_{c}(\varPhi_{c}\tilde{z})\in W_{c}$ it follows that
  \begin{equation*}
    \begin{aligned}
      \|\tilde{\psi}\|_{C^{1}}&\leq \|\varPhi_{c}\tilde{z}\|_{C^{1}}+\|w_{c}(\varPhi_{c}\tilde{z})\|_{C^{1}}\\
      &\leq (1+\sup_{\phi\in C_{c,0}}\|Dw_{c}(\phi)\|)\|\varPhi_{c}\tilde{z}\|_{C^{1}}\\
      &\leq (1+\sup_{\phi\in C_{c,0}}\|Dw_{c}(\phi)\|)\sum_{j=1}^{d}|\tilde{z}_{j}|\|\phi_{j}\|_{C^{1}}\\
      &\leq (1+\sup_{\phi\in C_{c,0}}\|Dw_{c}(\phi)\|)\max_{1\leq j\leq d}\|\phi_{j}\|_{C^{1}}\sum_{j=1}^{d}|\tilde{z}_{j}|\\
      &\leq k(1+\sup_{\phi\in C_{c,0}}\|Dw_{c}(\phi)\|)\max_{1\leq j\leq d}\|\phi_{j}\|_{C^{1}}\|\tilde{z}\|_{\R^{d}}\\
          &\leq k(1+\sup_{\phi\in C_{c,0}}\|Dw_{c}(\phi)\|)\max_{1\leq j\leq d}\|\phi_{j}\|_{C^{1}}\delta\\
          &<\delta_{d}.
    \end{aligned}
  \end{equation*}
  Therefore, $t_{+}(\tilde{\psi})=\infty$ and $\|F(t,\tilde{\psi})\|_{C^{1}}< \epsilon_{d}$ as $0\leq t<t_{+}(\tilde{\psi})$. Furthermore, in view of the choice of $\epsilon_{d}$, $F(t,\tilde{\psi})$ remains for all $t\geq 0$ in the open set $N_{c}=C_{c,0}+C_{su,0}^{1}$ where $W_{c}$ is positively invariant with respect $F$ (see property (ii) of $W_{c}$). As additionally $F(0,\tilde{\phi})\in W_{c}$ it follows that $F(t,\tilde{\psi})\in W_{c}$ for all $t\geq 0$. Thus, $F(t,\tilde{\psi})=P_{c}F(t,\tilde{\psi})+w_{c}(P_{c}F(t,\tilde{\psi}))$ as $t\geq 0$ and the $C^{1}$-smooth curve
  \begin{equation*}
    z^{\tilde{\psi}}:[0,\infty)\ni t\mapsto \varGamma_{c}P_{c}F(t,\tilde{\psi})\in\R^{d}
  \end{equation*}
  with initial value
  \begin{equation*}
    z^{\tilde{\psi}}(0)=\varGamma_{c}P_{c}F(0,\tilde{\psi})=(\varGamma_{c}\circ P_{c})(\varPhi_{c}\tilde{z}+w_{c}(\varPhi_{c}\tilde{z}))=\varGamma_{c}(\varPhi_{c}\tilde{z})=\tilde{z}
  \end{equation*}
  satisfies Eq. \eqref{eq: CM-reduction} for all $0\leq t<\infty$. From the uniqueness of solutions we conclude that the solution $z(\cdot;\tilde{z})$ of Eq. \eqref{eq: CM-reduction} with $z(0;\tilde{z})=\tilde{z}$ does exist for all $t\geq 0$ and coincide with the curve $z^{\tilde{\psi}}$. In particular,
  \begin{equation*}
    \|z(t,\tilde{z})\|_{\R^{d}}=\|z^{\tilde{\psi}}(t)\|_{\R^{d}}
    =\|\varGamma_{c}P_{c}F(t,\tilde{\psi})\|_{\R^{d}}\leq \|\varGamma P_{c}\|\|F(t,\tilde{\psi})\|_{C^{1}}<\|\varGamma P_{c}\|\epsilon_{d}<\epsilon
  \end{equation*}
  as $t\geq 0$, which finally finishes the proof of our claim and so of the proposition.
\end{proof}

\section{Example}
In this final section, we give a concrete example to illustrate the application of Theorem \ref{thm: linearized} and, especially, of Theorem \ref{thm: reduction}. For doing so, set $h=1$ and $n=1$ in the definitions of the Banach spaces $C$ and $C^{1}$, and consider the scalar differential equation
\begin{equation}\label{eq: exchange}
  x^{\prime}(t)=a[x(t)-x(t-r)]-|x(t)|x(t)
\end{equation}
with a real parameter $a>0$ and a delay $r>0$. This equation represents a mathe\-matical model to describe short-term fluctuations of exchange rates. Originally, it was motivated in the case of the constant delay $r=1$ and a thorough discussion of Eq. \eqref{eq: exchange} and the behavior of its solutions in this situation is contained in Brunovsk\'{y} et al. \cite{Brunovsky2004}. Here, we consider the situation of a state-dependent delay $r=r(x(t))>0$, that is,
\begin{equation}\label{eq: gen_exchange}
  x^{\prime}(t)=a[x(t)-x(t-r(x(t)))]-|x(t)|x(t),
\end{equation}
which is studied in \cite{Stumpf2010,Stumpf2012}. For the delay function $r:\R\to\R$ under consideration, it is assumed that the following hypotheses hold:
\begin{description}
  \item[(DF1)] $r$ is $C^{1}$-smooth,
  \item[(DF2)] $0<r(s)\leq r(0)=:r_{0}$ for all $s\in\R$,
  \item[(DF3)] $r(s)=r(-s)$ for all $s\in\R$, and
  \item[(DF4)] $r_{0}=1$.
\end{description}
Observe that different results in \cite{Stumpf2010,Stumpf2012} -- in particular, the main result in \cite{Stumpf2012} -- require the additional assumption
\begin{description}
  \item[(DF5)] $|r^{\prime}(s)|<1/(4a^2)$ for all $-2a\leq s\leq 2a$
\end{description}
on the delay function $r$, where the real $a>0$ is just the parameter involved in Eq. \eqref{eq: gen_exchange}. However, for the application of Theorem \ref{thm: linearized} as well as of Theorem \ref{thm: reduction} discussed in the following, the restriction (DF5) on $r$ is not needed. Therefore, unless otherwise stated, we consider Eq. \eqref{eq: gen_exchange} under the assumption that $r$ does only satisfy conditions (DF1)-(DF4), and begin our discussion with repeating some relevant material from \cite{Stumpf2010,Stumpf2012} without proofs below.

To begin with, observe that by the map
\begin{equation*}
  f:C^{1}\ni\phi\mapsto a[\phi(0)-\phi(-r(\phi(0)))]-|\phi(0)|\phi(0)\in\R
\end{equation*}
Eq. \eqref{eq: gen_exchange} takes the more abstract from
\begin{equation}\label{eq: gen_exchange_abs}
  x^{\prime}(t)=f(x_{t}).
\end{equation}
Obviously, $f(0)=0$, and it is also not hard to see that $f$ satisfies the smoothness conditions (S1) and (S2). In particular,  $X_{f}=\lbrace \psi\in C^{1}\mid \psi^{\prime}(0)=f(\psi)\rbrace$ is not empty and $x:\R\ni t\mapsto 0\in\R$ is a solution of Eq. \eqref{eq: gen_exchange_abs}. Moreover, for each $\phi\in X_{f}$ there is a uniquely determined solution $x^{\phi}:[-1,\infty)\to\R$ of Eq. \eqref{eq: gen_exchange_abs} with $x_{0}^{\phi}=\phi$ and $x_{t}^{\phi}\in X_{f}$ as $t\geq 0$. The relations $F(t,\phi)=x^{\phi}_{t}$, $0\leq t<\infty$ and $\phi\in X_{f}$, define a continuous semiflow $F:[0,\infty)\times X_{f}\to X_{f}$ with $C^{1}$-smooth time-$t$-maps $F_{t}:=F(t,\cdot)$, $t\geq 0$, and the stationary point $\phi_{0}=0\in X_{f}$.

Now, we are interested in the stability of the stationary point $\phi_{0}=0$ of the semiflow $F$ in dependence of the parameter $a>0$. For this reason, we write Eq. \eqref{eq: gen_exchange_abs} as
\begin{equation}
  x^{\prime}(t)=Lx_{t}+g(x_{t})
\end{equation}
with the linear operator
\begin{equation*}
  L=Df(0):C^{1}\ni\phi\mapsto a[\phi(0)-\phi(-1)]\in\R
\end{equation*}
and the non-linear map
\begin{equation*}
  g:C^{1}\ni\phi\mapsto f(\phi)-L\phi=a[\phi(-1)-\phi(-r(\phi(0)))]-|\phi(0)|\phi(0)\in\R.
\end{equation*}
The linear extension $L_{e}=D_{e}f(0)$ of the bounded linear operator $L:C^{1}\to\R$ to the greater Banach space $C$ is obviously given by
\begin{equation*}
  L_{e}=D_{e}f(0):C\ni\phi\mapsto a[\phi(0)-\phi(-1)]\in\R
\end{equation*}
and induces the linear retarded functional differential equation
\begin{equation}\label{eq: linearization_mod}
  v^{\prime}(t)=L_{e}v_{t}=a[v(t)-v(t-1)].
\end{equation}
For each $\psi\in C$ this linear equation has a unique solution $v^{\psi}:[-1,\infty)\to\R$ satisfying $v^{\psi}_{0}=\psi$. The associated solution semigroup $T_{e}:=\lbrace T_{e}(t)\rbrace_{t\geq 0}$ is defined by $T_{e}(t):C\ni\psi\mapsto v^{\psi}_{t}\in C$ and recall that it is closely related to the linearization $T:=\lbrace DF_{t}(0)\rbrace_{t\geq 0}$ of the semiflow $F$ at $\phi_{0}=0\in X_{f}$. In particular, we have $\sigma(G_{e})=\sigma(G)$ for the spectra of the generators $G_{e}$ and $G$ of the two semigroups $T_{e}$ and $T$, respectively. Using the ansatz $z(t)=e^{\lambda t}$ with $\lambda\in \mathbb{C}$ for a solution of Eq. \eqref{eq: linearization_mod}, we find the characteristic equation
\begin{equation}\label{eq: char}
\triangle (\lambda)=0
\end{equation}
with $\triangle(\lambda):=\lambda-a[1-e^{-\lambda}]$, and the set of roots of Eq. \eqref{eq: char} just coincides with $\sigma(G_{e})$ (or more precisely, with the spectrum of the complexification $(G_{e})_{\C}$ of the operator $G_{e}$). Furthermore, the order of a root $\lambda\in\C$ of Eq. \eqref{eq: char} agrees with the dimension of the generalized eigenspace of $G_{e}$ associated with $\lambda$. For the location of these roots in the complex plane one finds the following:
\begin{itemize}
  \item[(1)] If $a>0$ and if $a\not=1$, then $\lambda_{0}=0$ is a simple root of Eq. \eqref{eq: char}. In the situation $a=1$, $\lambda=0$ is a double root of Eq. \eqref{eq: char}.
  \item[(2)] For all $a>0$ with $a\not=1$, Eq. \eqref{eq: char} has a unique non-zero root $\lambda=\kappa\in\R$. The root $\lambda=\kappa$ is simple, and $\kappa<0$ for $0<a<1$ and $\kappa>0$ for $a>1$.
  \item[(3)] Apart from the real roots from (1) and (2), all other roots of Eq. \eqref{eq: char} for parameter $a>0$ occur in conjugate complex pairs $\mu\pm i\nu$ with $\mu<0$ and $\nu\not=0$.
\end{itemize}

By combining the statement (2) about $\sigma(G_{e})$ with part (i) of Theorem \ref{thm: linearized}, we immediately get our first stability result:
\begin{prop}\label{prop: example_unstable}
   Let the delay function $r$ satisfy the assumptions (DF1)-(DF4). Then for each $a>1$ the stationary point $\phi_{0}=0$ of the semiflow $F$, or equivalently, the zero solution of Eq. \eqref{eq: gen_exchange}, is unstable.
\end{prop}
Under the conditions (DF1)-(DF5) on $r$, the last result was already shown in \cite[Corollary 4.11]{Stumpf2010}. However, for $0<a\leq 1$ the application of Theorem \ref{thm: linearized} fails due to the presence of the zero root of Eq. \eqref{eq: char}, and the article \cite{Stumpf2010} contains only the conjecture -- compare page 109 in \cite{Stumpf2010} -- that in this situation the zero solution of Eq. \eqref{eq: gen_exchange} should be locally asymptotically stable. Below, we prove this conjecture at least for $0<a<1$ rigorously.

Given $0<a\not=1$, $\lambda_{0}=0$ is the only root of Eq. \eqref{eq: char} which lies on the imaginary axis of the complex plane, and $\lambda_{0}$ is simple. The center space $C_{c}\subset C^{1}$ has the dimension one and is spanned by the constant function $\eta_{0}:[-1,0]\ni \theta\mapsto 1\in\R$ in view of $G_{e}\eta_{0}=\eta^{\prime}_{0}=0=0\cdot \eta_{0}$. In particular, near the stationary point $\phi_{0}=0$ of $F$ there is a one-dimensional local center manifold $W_{c}=\lbrace \phi+w_{c}(\phi)\mid \phi\in C_{c,0}\rbrace$ given by a $C^{1}$-smooth map $w_{c}: C_{c,0}\to C_{u}\oplus C_{s}^{1}$ defined on some open neighborhood $C_{c,0}$ about $0$ in $C_{c}=\R\eta_{0}$ and satisfying both $w_{c}(0)=0$ and $Dw_{c}(0)=0$. Next, consider the reduction of $F$ to $W_{c}$. In this concrete example, the matrix $B_{c}\in\R^{1\times 1}$ from Eq. \eqref{eq: CM-reduction} is just the zero matrix. Hence, the center manifold reduction of $F$ to $W_{c}$ is given by
\begin{equation}\label{eq: CM-reduction-example}
  z^{\prime}(t)=h(z(t))
\end{equation}
with a continuously differentiable function $h:\R\supset V\to \R$ from an open interval $V$ containing $0\in\R$ into $\R$. We have $h(0)=0$ and $Dh(0)=0$. Moreover, as shown in \cite[Chapter 4.3]{Stumpf2010}, in close vicinity of $0\in\R$ the function $h$ has the asymptotic expansion
\begin{equation*}
  h(z)=-\frac{1}{1-a}|z|z+o(|z|^{2})
\end{equation*}
with the involved parameter $a$.
Consequently, in a sufficiently small neighborhood of the origin the center manifold reduction \eqref{eq: CM-reduction-example} reads
\begin{equation}\label{eq: CM-reduction-example-asym}
  z^{\prime}(t)=-\frac{1}{1-a}|z(t)|z(t)+o(|z(t)|^{2}).
\end{equation}
By combining this observation with Theorem \ref{thm: reduction}, we are now able to prove the local asymptotic stability of the stationary point $\phi_{0}$ of $F$ in case $0<a<1$.
\begin{prop}\label{prop: example_asym}
 Suppose that the delay function $r$ satisfies assumptions (DF1)-(DF4). Then for each $0<a<1$ the stationary point $\phi_{0}=0$ of $F$, and so the zero solution of Eq. \eqref{eq: gen_exchange}, is locally asymptotically stable.
\end{prop}
\begin{proof}
  First, observe that, under the given condition $0<a<1$, the coefficient of the leading term $|z(t)|z(t)$ of the asymptotic expansion on the right-hand side of Eq. \eqref{eq: CM-reduction-example-asym} is negative. Therefore, the function $\bar{V}:\R\ni z\mapsto (z^{2}/2)\in\R$ is positive definite whereas, in view of
  \begin{equation*}
    \bar{V}^{\prime}(z)\cdot h(z)=-\frac{1}{1-a}z^{2}|z|+o(|z^{3}|),
  \end{equation*}
  the orbital derivative of $\bar{V}$ along solutions of Eq. \eqref{eq: CM-reduction-example} is locally negative definite. Hence, we find some interval containing $0\in\R$ where $\bar{V}$ is a strict Lyapunov function for Eq. \eqref{eq: CM-reduction-example}. As, for instance, proven in Amann \cite[Chapter IV.18]{Amann1990}, it follows that the zero solution $z:\R\ni t\mapsto 0\in\R$ of Eq. \eqref{eq: CM-reduction-example} is locally asymptotically stable. But then Theorem \ref{thm: reduction} implies that the stationary point $\phi_{0}$ of $F$, or equivalently, the zero solution of Eq. \eqref{eq: gen_exchange}, is locally asymptotically stable as well. This finishes the proof.
\end{proof}

\begin{remark}
  1. In the case of the constant delay $r=1$, the statement of Proposition \ref{prop: example_asym} was proved in Brunovsk\'{y} et al. \cite[Corollary 5.1]{Brunovsky2004}.

  2. In the above proof we used a strict Lyapunov function. But that is by no means necessary. The key ingredient is the asymptotic expansion on the right-hand side of Eq. \eqref{eq: CM-reduction-example-asym} in combination with the observation that the coefficient of the leading term is negative. By starting from these, it is possible to carry out an elementary proof of the assertion without using any Lyapunov function.

  3. Of course, in the situation $a=1$ it is also possible to carry out a center manifold reduction including an asymptotic expansion of its right-hand side, in order to try to determine the stability property of $\phi_{0}=0$ by application of Theorem \ref{thm: reduction}. But observe that in the case $a=1$ the local center manifolds are two-dimensional as $\lambda_{0}$ is a double root of Eq. \eqref{eq: char}, and the stability analysis of the reduced differential equation seems to be much more difficult to access (compare Brunosk\'{y} et al. \cite{Brunovsky2004}).
\end{remark}

\end{document}